\documentclass{amsart}
\usepackage[margin=4cm]{geometry}
\usepackage{graphicx}
\usepackage{color}
\usepackage{enumerate}
\usepackage{tikz}
\usepackage{graphicx}
\usepackage{subcaption}
\usepackage{float}
\usepackage{comment}

\newtheorem{thm}{Theorem}[section]
\newtheorem{cor}[thm]{Corollary}
\newtheorem{lem}[thm]{Lemma}
\newtheorem{prop}[thm]{Proposition}
\theoremstyle{definition}
\newtheorem{defn}[thm]{Definition}
\theoremstyle{remark}
\newtheorem{rem}[thm]{Remark}

\numberwithin{equation}{section}

\begin{document}

\title[Eigenvalue bounds for the $p$-Laplacian]{Conformal upper bounds for the eigenvalues of the $p$-Laplacian\\
}%
\author{Bruno Colbois}%
\address{Bruno Colbois, Universit\'e de Neuch\^atel, Institute de Math\'ematiques, Rue Emile Argand 11, 2000 Neuch\^atel, Switzerland}%
\email{bruno.colbois@unine.ch}%
\author{Luigi Provenzano}%
\address{Luigi Provenzano, Sapienza Universit\`a di Roma, Dipartimento di Scienze di Base e Applicate per l'Ingegneria, Via Antonio Scarpa 16, 00161 Roma, Italy}%
\email{luigi.provenzano@uniroma1.it}%
\thanks{}%

\date{}%
\begin{abstract}
In this note we present upper bounds for the variational eigenvalues of the $p$-Laplacian on smooth domains of complete $n$-dimensional Riemannian manifolds and Neumann boundary conditions, and on compact (boundaryless) Riemannian manifolds. In particular, we provide upper bounds in the conformal class of a given manifold $(M,g)$ for $1<p\leq n$, and upper bounds for all $p>1$ when we fix a metric $g$. To do so, we use a metric approach for the construction of suitable test functions for the variational characterization of the eigenvalues. The upper bounds agree with the well-known asymptotic estimate of the eigenvalues due to Friedlander. We also present upper bounds for the variational eigenvalues on hypersurfaces bounding smooth domains in a Riemannian manifold in terms of the isoperimetric ratio.
\end{abstract}
\maketitle




\section{Introduction and statement of the main results}

Let $(M,g)$ be a complete, $n$-dimensional smooth Riemannian manifold, $n\geq 2$,  and let $\Omega\subseteq M$ be a bounded domain, i.e., a bounded connected open set, with boundary $\partial\Omega$. Let $p>1$. We consider the Neumann eigenvalue problem for the $p$-Laplace operator:

\begin{equation}\label{pLap}
\begin{cases}
-{\rm div}(|\nabla u|^{p-2}\nabla u)=\mu|u|^{p-2}u\,, & {\rm in\ }\Omega,\\
\frac{\partial u}{\partial\nu}=0\,, & {\rm in\ }\partial\Omega. 
\end{cases}
\end{equation}

Here $\frac{\partial u}{\partial\nu}$ denotes the derivative of $u$ in the direction of the exterior unit normal to the boundary, $\nu$. If $M$ is compact (boundaryless) and $\Omega=M$, then $\partial\Omega=\emptyset$ and we consider the closed problem $-{\rm div}(|\nabla u|^{p-2}\nabla u)=\mu|u|^{p-2}u$ in $M$.

Eigenvalue problems for the $p$-Laplacian with Dirichlet boundary conditions on Euclidean domains have been largely investigated in the last decades. We mention \cite{azorero_alonso} and especially \cite{lindqvist0} for extensive references on the subject. As for Neumann boundary conditions, less information is available. We refer to \cite{an_le} for an exhaustive presentation of eigenvalues problems for the $p$-Laplacian on Euclidean domains subject to various boundary conditions.

It is well-known that problem \eqref{pLap} admits an increasing sequence of non-negative eigenvalues obtained through the {\it Ljusternik-Schnirelman} principle
$$
0=\mu_{1,p}<\mu_{2,p}\leq\cdots\leq\mu_{k,p}\leq\cdots\nearrow+\infty,
$$
which are also called the {\it variational eigenvalues}. It is still an open problem whether other eigenvalues exist (except for the case $p=2$ or $n=1$). Results concerning the first positive eigenvalue and Neumann boundary conditions are available in \cite{elhabib,huang_p} for Euclidean domains. Concerning higher eigenvalues, results in the case of Dirichlet conditions for Euclidean domains can be found in \cite{parini1,parini2,parini3} (continuity and limits with respect to $p$, multiplicity, etc.). As for Riemannian manifolds, we mention \cite{colbois_plap,matei_first,matei_cheegerbuser,valtorta1,valtorta2} for sharp estimates for the first positive eigenvalue in case of compact manifolds or domains and Neumann boundary conditions, under a given lower bound on the Ricci curvature. We also mention \cite{matei_boundedness,matei} for estimates in a conformal class of a given metric for the first positive eigenvalue of a compact manifold. Results for the first eigenvalue of the corresponding Dirichlet problem on domains are available in \cite{mao1,mao2}. 

In this paper we investigate upper bounds for all variational eigenvalues that agree with the Weyl's law under suitable geometrical assumptions on the manifold $M$. In particular, we will investigate upper bounds in the conformal class of a given metric and upper bounds for a fixed metric. 

Results in this spirit are classical and well-known in the case of the Laplacian ($p=2$). For Euclidean domains we recall in particular the classical result of Kr\"oger \cite{kro} which is stated as follows:
\begin{equation}\label{kroger}
\mu_{k,2}\leq C_n\left(\frac{k}{|\Omega|}\right)^{\frac{2}{n}},
\end{equation}
for all $k\in\mathbb N$. Here $C_n=\left(\frac{2+n}{2}\right)^{2/n}4\pi^2\omega_n^{-2/n}$ and $\omega_n$ denotes the volume of the unit ball in $\mathbb R^n$. This result agrees with the Weyl's law for the Neumann Laplacian, namely
\begin{equation}\label{weylLap}
\lim_{k\rightarrow+\infty}\frac{\mu_{k,2}}{k^{\frac{2}{n}}} = 4\pi^2\omega_n^{-\frac{2}{n}}|\Omega|^{-\frac{2}{n}}.
\end{equation}
We note that the constant of the upper bounds \eqref{kroger} does not coincide with the constant in the Weyl's law. The validity of \eqref{kroger} with $C_n=4\pi^2\omega_n^{-2/n}$ is a famous open problem: the Polya's conjecture for the Neumann Laplacian (see \cite{polyaconj,polyatiling}). We also note that for the eigenvalues of the Neumann Laplacian, lower bounds for the eigenvalues of the form \eqref{kroger} do not hold in general, at least for the bottom of the spectrum. In fact it is well-known that one can produce domains with fixed volume and an arbitrary number of eigenvalues close to zero. Thus, for Neumann-type problems, lower bounds necessarily depend on the geometry of the domain in a more involved form (and not merely on the volume), see \cite{chavel,cheeger_lower} for a discussion on Cheeger lower bounds on domains and on compact manifolds, respectively. On the other hand, uniform upper bounds (depending on the dimension and the volume only) are relevant for Neumann-type problems. Upper bounds \eqref{kroger} are obtained in \cite{kro} by classical techniques of harmonic analysis and are easily extended to the Neumann eigenvalues of higher order linear elliptic operators on Euclidean domains, see \cite{Lap1997}.

As for the eigenvalues of the Laplacian on a compact $n$-dimensional Riemannian manifold $(M,g)$ with ${\rm Ric}_g\geq -(n-1)\kappa^2$, $\kappa\geq 0$, uniform upper bounds were proved by Buser \cite{buser}. Namely
\begin{equation}\label{buser}
\mu_{k,2}\leq \frac{(n-1)^2}{4}\kappa^2+C_n'\left(\frac{k}{|M|}\right)^{\frac{2}{n}},
\end{equation}
where $C_n'$ depends only on the dimension. We note that an additive term appears in the estimate. When ${\rm Ric}_g$ is supposed to be non-negative, then we have an analogous estimate as \eqref{kroger}, however the presence of an additive term of the form $\frac{(n-1)^2}{4}\kappa^2$ is necessary, see \cite{brooks0}. Again, these upper bounds agree with the Weyl's law which is given for a compact manifold by \eqref{weylLap} with $\Omega=M$. We also note that the geometry of the manifold enters the bounds as an {\it additive} constant, while the term encoding the asymptotic behavior, as expected, contains only information on the dimension and the volume of the manifold. 

Later, in \cite{colbois_maerten} Colbois and Maerten proved the analogue of \eqref{buser} for the Neumann eigenvalues of the Laplacian on domains $\Omega$ of complete $n$-dimensional Riemannian manifolds with ${\rm Ric}_g\geq -(n-1)\kappa^2$, $\kappa\geq 0$:
\begin{equation}\label{colbois}
\mu_{k,2}\leq A_n\kappa^2+B_n\left(\frac{k}{|\Omega|}\right)^{\frac{2}{n}}.
\end{equation}
Inequalities \eqref{buser} and \eqref{colbois} are valid when a metric $g$ is fixed. 

Another relevant problem is to provide bounds within the conformal class $[g]$ of a given metric $g$ on a compact manifold $M$. The first results in this sense for the eigenvalues of the Laplacian are due to \cite{yang_yau} for $n=2$ and \cite{korevaar} for $n\geq 2$ (see also \cite{EI} for results concerning the first positive eigenvalue). The result of \cite{korevaar} can be summarized as follows: for any compact $n$-dimensional Riemannian manifold $(M,g)$
\begin{equation}\label{korevaar}
\mu_{k,2}|M|^{\frac{2}{n}}\leq C_n([g]) k^{\frac{2}{n}},
\end{equation}
for all $k\in\mathbb N$, where $C_n([g])$ depends only on $n$ and on the conformal class of $g$.

Inequality \eqref{korevaar} has been then improved in \cite{asma_conf}. In particular, in \cite{asma_conf} the following inequality is proved for any compact $n$-dimensional Riemannian manifold $(M,g)$
\begin{equation}\label{asma}
\mu_{k,2}|M|^{\frac{2}{n}}\leq A_n V([g])^{\frac{2}{n}}+B_n k^{\frac{2}{n}},
\end{equation}
for all $k\in\mathbb N$, where $V([g])$ denotes the min-conformal volume (see \eqref{min-conf-vol} below for the definition). Inequalities \eqref{buser} and \eqref{colbois} can be deduced by \eqref{asma} when we fix a metric. In particular, in \cite{asma_conf} an analogous inequality as \eqref{asma} is proved for the Neumann eigenvalues on domains, which implies the bound \eqref{colbois}.

Analogous estimates, all in the same spirit of \cite{buser,colbois_maerten,asma_conf} have been proved for the Steklov problem \cite{colboisgirouard_steklov}, for  hypersurfaces \cite{colboisgirouard}, and for the Neumann eigenvalues of the biharmonic operator on domains of Riemannian manifolds with Ricci curvature bounded from below \cite{colbois_provenzano_2}.

The classical approach to prove upper bounds for the eigenvalues of linear elliptic operators on Euclidean domains, as mentioned, relies on harmonic analysis. However this approach is no more suitable for compact manifolds or, in general, for domains of complete manifolds. The approach which has been used in \cite{colbois_maerten} (which is in the same spirit of \cite{buser}, see also \cite{gny}) makes use of a metric construction. Namely, in order to bound $\mu_{k,2}$ one considers $A_1,...,A_k$ disjoints subsets of $\Omega$ of measure of
the order of $\frac{|\Omega|}{k}$, and introduce test functions $u_1,...,u_k$ subordinated to these sets. A clever estimate of the Rayleigh quotient of these functions provides the upper bounds \eqref{colbois} or \eqref{asma} (depending on the particular construction).

As for the variational eigenvalues of the Neumann $p$-Laplacian with $p\ne 2$ much less is known. Asymptotic estimates are available, which are consequence, as in the case of the Laplacian, of the Weyl's law. Actually, the validity of a Weyl's law for the Neumann (and Dirichlet) variational eigenvalues of the $p$-Laplacian has been conjectured by Friedlander in \cite{friedlander_p}, who proved an asymptotic estimate of the form $C_1\lambda^{n/p}\leq N(\lambda)\leq C_2\lambda^{n/p}$ with $C_1,C_2$ depending only on $n,p$, as $\lambda\rightarrow+\infty$, for the counting function $N(\lambda):=\sharp\{k:\mu_{k,p}\leq\lambda\}$. The conjecture $C_1=C_2$ has been recently proved in \cite{mazurowski} (see also \cite{iannizzotto_squassina} for asymptotic estimates for the eigenvalues of the fractional $p$-Laplacian).

In order to state our main results we need to fix some notation and introduce some definitions. Through the rest of the paper, we will denote by $\mu_{k,p}^g$ either the Neumann eigenvalues of the $p$-Laplacian on a domain $\Omega$ of a complete Riemannian manifold $(M,g)$, or the eigenvalues of the $p$-Laplacian on a compact Riemannian manifold $(M,g)$, unless otherwise specified. For a complete Riemannian manifold $(M,g)$, we shall denote by $|\cdot|_g$ the Riemannian measure associated with the metric $g$. By $[g]$ we denote the conformal class of a given metric $g$.

The first main result of the paper is the following.
\begin{thm}\label{main}
Let $(M,g_0)$ be a complete $n$-dimensional Riemannian manifold, $n\geq 2$, with Ricci curvature bounded below ${\rm Ric}_{g_0}\geq -(n-1)\kappa^2$, $\kappa\geq 0$ and let $1<p\leq n$. Then, for all bounded domains $\Omega\subset M$ with smooth boundary and all metric $g\in[g_0]$ there exist constants $A_{p,n}, B_{p,n}>0$ only depending on $p$ and $n$ such that
\begin{equation}\label{main_ineq}
\mu_{k,p}^g|\Omega|_g^{\frac{p}{n}}\leq A_{p,n}\kappa^p |\Omega|_{g_0}^{\frac{p}{n}}+B_{p,n}k^{\frac{p}{n}},
\end{equation}
for all $k\in\mathbb N$.
\end{thm}

Theorem \ref{main} can be stated also for compact manifolds. To do so, we need the following definition (see also \cite{asma_conf})
\begin{defn}\label{min-conf-vol}
Let $(M,g)$ be a compact $n$-dimensional Riemannian manifold. We define the min-conformal volume $V([g])$ by
$$
V([g]):=\inf\left\{|M|_{g_0}: g_0\in[g], {\rm Ric}_{g_0}\geq-(n-1)\right\}
$$
\end{defn}
It is a standard to verify that, if we denote by $\kappa(g)$ the smallest non-negative number such that ${\rm Ric}_{g}\geq -(n-1)\kappa(g)^2$, then
$$
V([g]):=\inf\left\{|M|_{g_0}\kappa(g_0)^{n}: g_0\in[g]\right\}
$$
We have the following theorem.

\begin{thm}\label{cor}
For all $1<p\leq n$, $n\geq 2$, there exist constants $A_{p,n}, B_{p,n}>0$ only depending on $p$ and $n$, such that for all compact $n$-dimensional Riemannian manifolds and all $k\in\mathbb N$ we have
\begin{equation}\label{main_ineq_cor}
\mu_{k,p}^g|M|_g^{\frac{p}{n}}\leq A_{p,n}V([g])^{\frac{p}{n}}+B_{p,n}k^{\frac{p}{n}}.
\end{equation}
\end{thm}

The proof of Theorem \ref{cor} follows easily from that of Theorem \ref{main} (see Remark \ref{remark_manifold}).

We note that inequalities of the type \eqref{main_ineq}-\eqref{main_ineq_cor} are not possible for $p>n$. In fact, it is proved in \cite{matei} that if $p>n$, for any $n$-dimensional Riemannian manifold $(M,g_0)$ there exists a metric in $g\in [g_0]$ and of volume one with $\mu_{2,p}^g$ arbitrarily large.

However, if we fix the metric, we have upper bounds for any $p>1$. The second main result of this paper is the following.

\begin{thm}\label{main_2}
Let $(M,g)$ be a complete $n$-dimensional Riemannian manifold with Ricci curvature bounded below ${\rm Ric}_g\geq -(n-1)\kappa^2$, $\kappa\geq 0$, and let $p>1$. Then, for all bounded domains $\Omega\subset M$ with smooth boundary there exist constants $A_{p,n}, B_{p,n}>0$ only depending on $p$ and $n$ such that
\begin{equation}\label{main_ineq_2}
\mu_{k,p}^g\leq A_{p,n}\kappa^p+B_{p,n}\left(\frac{k}{|\Omega|_g}\right)^{\frac{p}{n}},
\end{equation}
for all $k\in\mathbb N$.
\end{thm}

As in the case of Theorem \ref{main}, Theorem \ref{main_2} can be stated also for compact manifolds (see Remark \ref{remark_manifold_2}).

\begin{thm}\label{cor_2}
For all $p>1$, $n\geq 2$, there exist constants $A_{p,n}, B_{p,n}>0$ only depending on the $p$ and $n$, such that for all compact $n$-dimensional Riemannian manifolds  with Ricci curvature bounded below ${\rm Ric}_g\geq -(n-1)\kappa^2$, $\kappa\geq 0$ we have
\begin{equation}\label{main_ineq_cor_2}
\mu_{k,p}^g\leq A_{p,n}\kappa^p+B_{p,n}\left(\frac{k}{|M|_g}\right)^{\frac{p}{n}},
\end{equation}
for all $k\in\mathbb N$.
\end{thm}

Theorem \ref{cor} carries a number of corollaries. In particular, as in \cite{colboisconformal} it is possible to introduce the {\it $p$-conformal variational eigenvalues} which are defined by
\begin{equation}
\tilde\mu_{k,p}^g:=\sup\left\{\mu_{k,p}^{g'}|M|_{g'}^{\frac{p}{n}}:g'\in[g]\right\}.
\end{equation}
Theorem \ref{cor} implies the following bounds on the $p$-conformal variational eigenvalues.

\begin{cor}\label{cor1}
For all $1<p\leq n$, $n\geq 2$, there exist constants $A_{p,n}, B_{p,n}>0$ only depending on $p$ and $n$, such that for all compact $n$-dimensional Riemannian manifolds and all $k\in\mathbb N$ we have
\begin{equation}\label{main_ineq_cor_pconf}
\tilde\mu_{k,p}^g\leq A_{p,n}V([g])^{\frac{p}{n}}+B_{p,n}k^{\frac{p}{n}}.
\end{equation}
\end{cor}

In the case of surfaces (i.e., $n=2$) we know that any orientable Riemannian $2$-manifold of genus $\gamma$ is conformally equivalent to the standard sphere of constant curvature $1$ if $\gamma=0$, to the flat torus if $\gamma=1$, and to a manifold with constant curvature $-1$ if $\gamma\geq 2$ (Uniformization Theorem). In the first two cases $V([g])=0$, while in the third case $V([g])\leq 4\pi(\gamma-1)$. From this and Theorem \ref{cor} we deduce the following corollary.

\begin{cor}\label{cor2}
For all $1<p\leq 2$, there exist constants $A_{p}, B_{p}>0$ only depending on $p$, such that for all compact orientable Riemannian surfaces $(\Sigma_{\gamma},g)$ of genus $\gamma$ and all $k\in\mathbb N$ we have
\begin{equation}\label{main_ineq_cor_surf}
\mu_{k,p}^g|\Sigma_{\gamma}|_g^{\frac{p}{2}}\leq A_{p}(\max\{0,\gamma-1\})^{\frac{p}{2}}+B_{p}k^{\frac{p}{2}}.
\end{equation}
\end{cor}

Finally, we want to remark that Theorem \ref{main_2} holds in particular for all Euclidean domains.

\begin{cor}\label{cor3}
For all bounded domains $\Omega$ in $\mathbb R^n$, $n\geq 2$, with smooth boundary and all $p>1$ there exist constants $A_{p,n}, B_{p,n}>0$ only depending on $p$ and $n$ such that
\begin{equation}\label{main_ineq_2_euclidean}
\mu_{k,p}\leq B_{p,n}\left(\frac{k}{|\Omega|}\right)^{\frac{p}{n}},
\end{equation}
for all $k\in\mathbb N$, where $\mu_{k,p}$ denotes the $k$-th variational eigenvalue of the Neumann $p$-Laplacian on $\Omega$ and $|\Omega|$ denotes the Lebesgue measure of $\Omega$.
\end{cor}

We note that in the Euclidean case it is not possible to prove \eqref{main_ineq_2_euclidean} as in \cite{kro,Lap1997} (except for the case $p=2$), where the Hilbert structure of $L^2(\Omega)$ is deeply involved in the proof.



The proofs of Theorems \ref{main} and \ref{main_2} are given in the same spirit of \cite{colbois_maerten, colbois_provenzano_2,asma_conf}. In fact, in this paper we formalize a general metric approach suitable to prove upper bounds for variational eigenvalues which has been refined more and more starting from \cite{gny}, passing through \cite{colbois_maerten,asma_conf} and ending with \cite{colbois_provenzano_2} and with the present note. 

Not only we formalize this metric approach stating a general decomposition result (Theorem \ref{corollary_small_annuli}): we also show how it can be exploited in two different ways, whose differences may seem slight but are substantial. A first way  allows to provide conformal upper bounds for $1<p\leq n$ (see the proof of Theorem \ref{main}), while an alternative way allows to prove upper bounds for all $p>1$ when we fix a metric $g$ (see the proof of Theorem \ref{main_2}). We note that the main decomposition result, Theorem \ref{corollary_small_annuli}, is a clever merging of a classical decomposition result of a metric measure space by capacitors (Theorem \ref{gny}, see also \cite[Theorem 1.1]{gny}) and of a refinement of a method introduced in \cite{colbois_maerten} (Lemma \ref{genCM-cor}). Actually, one can see that the proof of Theorem \ref{main_2} can be obtained just by exploiting Lemma \ref{genCM-cor}, i.e., morally it can be obtained with (a refinement of) the technique introduced in \cite{colbois_maerten}. Therefore, we deduce a sort of ``metatheorem'' stating that, if we are able to prove a result for the eigenvalues of the Laplacian with a method in the spirit of \cite{colbois_maerten}, the result is likely to hold also for the variational eigenvalues of the $p$-Laplacian. On the other hand, this is no more true if the classical decomposition of \cite{gny} is used.

We mention that a similar behavior has been observed for upper bounds on the Neumann eigenvalues of linear elliptic operators of order $2m$, $m\in\mathbb N$ and density on Euclidean domains, see \cite{colbois_provenzano}. In particular, if $2\leq 2m\leq n$ uniform upper bounds hold, which morally correspond to the conformal upper bounds discussed in this paper. For $n<2m$ these bounds do not hold (counter-examples are provided in \cite{colbois_provenzano}).


We also recall that Lemma \ref{genCM-cor} has been exploited to prove upper bounds for the eigenvalues on $\Sigma$, where $\Sigma$ is an hypersurface in a complete $n$-dimensional Riemannian manifold $(M,g)$ bounding some smooth domain $\Omega$. In fact, upper bounds in term of the isoperimetric ratio for the Laplacian eigenvalues on $\Sigma$ are obtained in \cite{colboisgirouard}. As already mentioned, when we have a result obtained by means of Lemma \ref{genCM-cor}, that result is likely to hold also for the variational eigenvalues of the $p$-Laplacian. This is in fact the case. We include the precise statements of the analogous results for the variational eigenvalues of the $p$-Laplacian, along with their proofs, in the Appendix \ref{appendix}. 

The paper is organized as follows. In Section \ref{pre} we set the notation and recall some preliminary results. In Section \ref{sub_metric} we collect all the
main technical results of decomposition of a metric measure space by capacitors. In Section \ref{bounds} we prove Theorems \ref{main} and \ref{main_2}. In Appendix \ref{appendix} we discuss upper bounds on hypersurfaces in terms of the isoperimetric ratio.




\section{Preliminaries and notation}\label{pre}

By $W^{1,p}(\Omega)$ we denote the Sobolev space of functions $u\in L^p(\Omega)$ with weak first derivatives in $L^p(\Omega)$. The space $W^{1,p}(\Omega)$ is endowed with the norm

\begin{equation}\label{Wnorm}
\|u\|_{W^{1,p}(\Omega)}^p:=\int_{\Omega}|\nabla_g u|^p+|u|^p dv_g,
\end{equation}
where $\nabla_g$ denotes the gradient associated with the metric $g$ and $dv_g$ denotes the Riemannian volume element associated with $g$. For $u\in L^p(\Omega)$ we denote by $\|u\|_{L^p(\Omega)}$ its standard norm given by
\begin{equation}\label{Lnorm}
\|u\|_{L^{p}(\Omega)}^p:=\int_{\Omega}|u|^p dv_g,
\end{equation}

Problem \eqref{pLap} is understood in the weak sense, namely a couple $(u,\mu)\in W^{1,p}(\Omega)\times\mathbb R$ is a weak solution to \eqref{pLap} if and only if
\begin{equation}\label{pLap_weak}
\int_{\Omega}|\nabla_g u|^{p-1}\nabla_g u\cdot\nabla_g\phi dv_g=\mu\int_{\Omega}|u|^{p-2}u\phi dv_g\,,\ \ \ \forall\phi\in W^{1,p}(\Omega).
\end{equation}

A sequence of eigenvalues for \eqref{pLap_weak} can be obtained through the {\it Ljusternik-Schnirelman} (see \cite{an_le} for a detailed discussion). These eigenvalues, which form an increasing sequence of non-negative numbers diverging to $+\infty$, are called the {\it variational eigenvalues} as they can be characterized variationally as follows:
\begin{equation}\label{minmax}
\mu_{k,p}^g:=\inf_{F\in\Gamma_k}\sup_{u\in F}\mathcal R_p(u),
\end{equation}
where
\begin{equation}\label{Rayleigh}
\mathcal R_p(u):=\frac{\int_{\Omega}|\nabla_g u|^pdv_g}{\int_{\Omega}|u|^pdv_g}
\end{equation}
is the Rayleigh quotient of $u$. Here

\begin{equation}\label{Gammak}
\Gamma_k:=\left\{F\subset W^{1,p}(\Omega)\setminus\left\{0\right\}: F\cap\left\{u:\|u\|_{L^p(\Omega)}=1 \right\}{\rm\ compact,\ }F{\rm\ symmetric,\ }\gamma(F)\geq k\right\},
\end{equation}

and $\gamma(F)$ denotes the {\it Krasnoselskii genus} of $F$, which is defined by

\begin{equation}\label{genus}
\gamma(F):=\min\left\{\ell\in \mathbb N:{\rm\ there\ exists\ }f:F\rightarrow\mathbb R^{\ell}\setminus\left\{0\right\}{\rm\ continuous\ and\ odd}\right\}.
\end{equation}

In order to prove upper bounds for $\mu_{k,p}^g$ we need suitable sets $F_k\in \Gamma_k$ to test in \eqref{minmax}. The following lemma provides us a useful way to build such $F_k$.

\begin{lem}\label{disjoint_lem}
Let $k\in\mathbb N$, $k\geq 1$, and let $u_1,...u_k\in W^{1,p}(\Omega)$, with $u_i\ne 0$ and with pairwise disjoint supports $U_1, ..., U_k$. Let
$$
F_k:=\left\{\sum_{i=1}^k\alpha_iu_i: \alpha_i\in\mathbb R, \sum_{i=1}^k|\alpha_i|^p=1\right\}.
$$
Then $F_k\in\Gamma_k$.
\end{lem}
\begin{proof}
Clearly $0\notin F_k$. Moreover, $F_k$ is symmetric and $F_k\cap\left\{u:\|u\|_{L^p(\Omega)=1} \right\}$ is compact. We show now that $\gamma(F_k)=k$. We define a map $f_k:F_k\rightarrow\mathbb R^k\setminus\left\{0\right\}$ by setting, for $u\in F_k$, $u=\sum_{i=1}^k\alpha_i u_i$,
$$
f_k(u)=\sum_{i=1}^k\alpha_i e_i^k,
$$
where $e_i^k$, $i=1,...,k$, denotes the standard basis of $\mathbb R^k$. 
The function $f_k$ is an  odd homeomorphism between $F_k$ and $\mathbb S^{k-1}_p:=\left\{x\in\mathbb R^k:\sum_{i=1}^k|x_i|^p=1\right\}$, which is the unit sphere of $\mathbb R^k$ with respect to the $\ell^p$ norm. This implies that $\gamma(F_k)=\gamma(\mathbb S^{k-1}_p)$ (see also \cite[Proposition 2.3]{szulkin}). Finally, by the Borsuk-Ulam Theorem we deduce that $\gamma(\mathbb S^{k-1}_p)=k$.
\end{proof}




\section{Decomposition of a metric measure space by capacitors}\label{sub_metric}

In this section we present the main technical tools which will be used to prove upper bounds for eigenvalues. We start with some definitions. 

We denote by $(X,{\rm dist},\varsigma)$ a metric measure space with a metric ${\rm dist}$ and a Borel measure $\varsigma$. We will call {\it capacitor} every couple $(A,D)$ of Borel sets of $X$ such that $A\subset D$. By an annulus in $X$ we mean any set $A\subset X$ of the form
\begin{equation*}
A=A(a,r,R)=\left\{x\in X:r<{\rm dist}(x,a)<R\right\},
\end{equation*}
where $a\in X$ and $0\leq r<R<+\infty$. By $2A$ we denote 
\begin{equation*}
2A=2A(a,r,R)=\left\{x\in X:\frac{r}{2}<{\rm dist}(x,a)<2R\right\}.
\end{equation*}
Moreover, for any $F\subset X$ and $r>0$ we denote the $r$-neighborhood of $F$ by $F^r$, namely
$$
F^r:=\left\{x\in X:{\rm dist}(x,F)<r\right\}.
$$

The prototype of the construction of a decomposition of a metric measure space by disjoint sets is given in \cite{gny}. We recall it here for the reader's convenience.

\begin{thm}[{\cite[Theorem 1.1]{gny}}]\label{gny}
Let $(X,{\rm dist},\varsigma)$ be a metric-measure space with $\varsigma$ a non-atomic finite  Borel measure. Assume that the following properties are satisfied:
\begin{enumerate}[i)]
\item there exists a constant $\Gamma$ such that any metric ball of radius $r$ can be covered by at most $\Gamma$ balls of radius $\frac{r}{2}$;
\item all metric balls in $X$ are precompact sets.
\end{enumerate}
Then for any integer $k$ there exists a sequence $\left\{A_i\right\}_{i=1}^k$ of $k$ annuli in $X$ such that, for any $i=1,...,k$
\begin{equation*}
\varsigma(A_i)\geq c\frac{\varsigma(X)}{k},
\end{equation*}
and the annuli $2A_i$ are pairwise disjoint. The constant $c$ depends only on the constant $\Gamma$ in i).
\end{thm}

Theorem \ref{gny} provides a decomposition of a metric measure space by annuli of the size at least $c\frac{\varsigma(X)}{k}$. The common idea of the proofs of Theorems \ref{main} and \ref{main_2} is to build for each $k\in\mathbb N$, suitable test functions $u_i$ supported on $2A_i$ and such that $u_i=1$ on $A_i$, and then to compute their Rayleigh quotients. In principle, in order to obtain the estimate of Theorems \ref{main} and \ref{main_2} with this method, we need that the constant $c$ which controls the volume of the $A_i$'s depends only on $n$. This is true, by the Bishop-Gromov volume comparison Theorem, in the case ${\rm Ric}_g\geq 0$, but this is no longer true in the case of a negative lower bound on the Ricci curvature. The direct application of Theorem  \ref{gny} would lead to an estimate of the form \eqref{main_ineq}-\eqref{main_ineq_2}, but with $B_{p,n}$ depending also on the diameter of $\Omega$ (see also \cite{gny} for more details), and such an estimate is bad: the geometry of the domain enters as a multiplicative constant in front of the Weyl-term. Moreover, the diameter should not affect the upper bounds.

We state now the following lemma, which improves \cite[Lemma 4.1]{colboisgirouard}. This Lemma will be useful to construct a decomposition of a metric measure space with capacitors alternative to Theorem \ref{gny}. The proof of this Lemma can be found in \cite{colbois_provenzano_2}.

\begin{lem}\label{genCM}
Let $(X,{\rm dist},\varsigma)$ be a compact metric measure space with a finite measure $\varsigma$. Assume that for all $s>0$ there exists an integer $N(s)$ such that each ball of radius $5s$ can be covered by $N(s)$ balls of radius $s$. Let $\beta>0$ satisfying $\beta\leq \frac{\varsigma(X)}{2}$ and let $r>0$ be such that for all $x\in X$
$$
\varsigma(B(x,r))\leq\frac{\beta}{2N(r)}.
$$
Then there exist two open sets $A$ and $D$ of $X$ with $A\subset D$ such that:
\begin{enumerate}[i)]
\item $A=B(x_1,r)\cup\cdots\cup B(x_l,r)$ with ${\rm dist}(x_i,x_j)\geq 4r$ if $i\ne j$;
\item $D=A^{4r}=B(x_1,5r)\cup\cdots\cup B(x_l,5r)$;
\item $\varsigma(A)\geq\frac{\beta}{2N(r)}$, $\varsigma(D)\leq\beta$ and ${\rm dist}(A,D^c)\geq 4r$.
\end{enumerate}
\end{lem}

A consequence of Lemma \ref{genCM} is the following result providing a decomposition of a metric measure space by capacitors which is alternative to that of Theorem \ref{gny}.

\begin{lem}\label{genCM-cor}
Let $(X,{\rm dist},\varsigma)$ be a compact metric measure space with a finite measure $\varsigma$. Assume that for all $s>0$ there exists an integer $N(s)$ such that each ball of radius $5s$ can be covered by $N(s)$ balls of radius $s$. If there exists an integer $k>0$ and a real number $r>0$ such that, for each $x\in X$
$$
\varsigma(B(x,r))\leq\frac{\varsigma(X)}{4N(r)^2k},
$$
then there exist $k$ $\varsigma$-measurable subsets $A_1,...,A_k$ of $X$ such that
$$
\varsigma(A_i)\geq\frac{\varsigma(X)}{2N(r)k},
$$
for all $i\leq k$, ${\rm dist}(A_i,A_j)\geq 4r$ for $i\ne j$, and 
$$
A_i=B(x^i_1,r)\cup\cdots\cup B(x^i_{l_i},r).
$$
\end{lem}

The proof of Lemma \ref{genCM-cor} is a consequence of Lemma \ref{genCM} and follows exactly the same lines of the proof of \cite[Lemma 2.1]{colboisgirouard}. We remark that \cite[Lemma 2.1]{colboisgirouard} provides a decomposition of a metric measure space by capacitors given by union of balls. In Lemma \ref{genCM-cor} the decomposition is given by unions of {\it disjoint} balls.

A clever merging of Theorem \ref{gny} and Lemma \ref{genCM-cor} allows to obtain the following Theorem, which provides a further construction of disjoint families of capacitors. This is the construction that we will use in the proofs of Theorems \ref{main} and \ref{main_2}. Its proof follows exactly the same lines as those of \cite[Theorem 2.1]{asma_conf}. In fact, the substantial difference is the use of Lemma \ref{genCM-cor} instead of \cite[Lemma 2.3]{asma_conf} (see also \cite[Lemma 2.1]{colboisgirouard} and \cite[Corollary 2.3]{colbois_maerten}).

\begin{thm}\label{corollary_small_annuli}
Let $(X,{\rm dist},\varsigma)$ be a compact metric-measure space with $\varsigma$ a non-atomic finite  Borel measure and let $a>0$. Assume that there exists a constant $\Gamma$ such that any metric ball of radius $0<r\leq a$ can be covered by at most $\Gamma$ balls of radius $\frac{r}{2}$. Then, for every $k\in\mathbb N$ there exists two families $\left\{A_i\right\}_{i=1}^k$ and  $\left\{D_i\right\}_{i=1}^k$ of Borel subsets of $X$ such that $A_i\subset D_i$, with the following properties:
\begin{enumerate}[i)]
\item $\varsigma(A_i)\geq c\frac{\varsigma(X)}{k}$, where $c$ depends only on $\Gamma$;
\item $D_i$ are pairwise disjoint;
\item the two families have one of the following form:
\begin{enumerate}[a)]
\item all the $A_i$ are annuli and $D_i=2A_i$, with outer radii smaller than $a$, or
\item all the $A_i$ are of the form $A_i=B(x_1^i,r_0)\cup\cdots\cup B(x_{l_i}^i,r_0)$, $G_i=A_i^{4r_0}$ and ${\rm dist}(x_k^i,x_l^i)\geq 4r_0$, where $r_0=\frac{4a}{1600}$.
\end{enumerate}
\end{enumerate}
\end{thm}

We remark that for a sufficiently large integer $k$ it is always possible to apply the construction of Theorem \ref{gny} and obtain a decomposition of the metric measure space by annuli (Theorem \ref{corollary_small_annuli} $i)$,$ii)$ and $iii)$-$a)$). In particular we have the following.

\begin{lem}\label{large_n}
Assume that the hypothesis of Theorem \ref{corollary_small_annuli} hold. Then there exists an integer $k_X$ such that for every $k\geq k_X$ there exists two families $\left\{A_i\right\}_{i=1}^k$ and  $\left\{D_i\right\}_{i=1}^k$ of Borel subsets of $X$ such that $A_i\subset D_i$ satisfying $i)$,$ii)$ and $iii)$-$a)$ of Theorem \ref{corollary_small_annuli}.
\end{lem}
We refer to \cite[Proposition 2.1]{asma_conf} for the proof of Lemma \ref{large_n}.

We state now a useful corollary of Theorem \ref{gny} which gives a lower bound of the inner radius of the annuli of the decomposition, see \cite[Remark\,3.13]{gny}.
\begin{cor}\label{corollarygny0}
Let the assumptions of Theorem \ref{gny} hold. Then each annulus $A_i$ has either internal radius $r_i$ such that
\begin{equation}\label{gny-rad-est}
r_i\geq\frac{1}{2}\inf\left\{r\in\mathbb R:V(r)\geq v_k\right\},
\end{equation}
where $V(r):=\sup_{x\in X}\varsigma(B(x,r))$ and $v_k=c\frac{\varsigma(X)}{k}$ , or is a ball of radius $r_i$ satisfying \eqref{gny-rad-est}.
\end{cor}
It turns out that Corollary \ref{corollarygny0} applies to the case $iii)$-$a)$ of Theorem \ref{corollary_small_annuli}, see also \cite{asma_conf}.

\section{Proof of Theorems \ref{main} and \ref{main_2}}\label{bounds}

In this section we present the proofs of Theorems \ref{main} and \ref{main_2}.

\begin{proof}[Proof of Theorem \ref{main}]
We apply Theorem \ref{corollary_small_annuli} with $a=\frac{1}{\kappa}$ (if $\kappa=0$ we take $a=+\infty$). We take $X=\Omega$ endowed with the  Riemannian distance induced by $g_0$, and with the measure $\varsigma_g$ defined as the restriction to $\Omega$ of the Lebesgue measure of $M$ associated with the metric $g$, namely $\varsigma_g(E)=|E\cap\Omega|_g$ for all measurable sets $E$.

\noindent{\bf Step 1 (large $k$).} From Lemma \ref{large_n} we deduce that there exists $k_{\Omega}\in\mathbb N$ such that for all $k\geq k_{\Omega}$ there exists a sequence $\left\{A_i\right\}_{i=1}^{3k}$ of $3k$ annuli such that $2A_i$ are pairwise disjoint and 
\begin{equation}\label{ineq_gen_1}
|\Omega\cap A_i|_g\geq c\frac{|\Omega|_g}{3k}.
\end{equation}
The constant $c$ depends only on $\Gamma$ of Theorem \ref{corollary_small_annuli}, hence it depends only on the dimension $n$ and can be determined explicitly (see \cite{asma_conf}). Since we have $3k$ annuli, we can pick at least $k$ of them such that
\begin{equation}\label{ineq_gen_2}
|\Omega\cap 2A_i|_g\leq \frac{|\Omega|_g}{k}.
\end{equation}
We take this family of $k$ annuli, and denote it by $\left\{A_i\right\}_{i=1}^k$.


Subordinated to this decomposition  we construct a family of $k$ disjointly supported functions $u_1,...,u_k\in W^{1,p}(\Omega)$, with ${\rm supp}(u_i)=2A_i$. Let us take
$$
F_k:=\left\{\sum_{i=1}^k\alpha_iu_i:\alpha_i\in\mathbb R,\sum_{i=1}^k|\alpha_i|^p=1\right\}.
$$

From \eqref{minmax} and from Lemma \ref{disjoint_lem} we deduce that
\begin{equation}\label{minmax12}
\mu_{k,p}^g\leq\sup_{u\in F_k}\frac{\int_{\Omega}|\nabla_g u|^pdv_g}{\int_{\Omega}|u|^pdv_g},
\end{equation}
which in particular implies, since $u_i$ are disjointly supported, that

\begin{equation}\label{minmax2}
\mu_{k,p}^g\leq\max_{i=1,...,k}\frac{\int_{\Omega}|\nabla u_i|_g^pdv_g}{\int_{\Omega}|u_i|^pdv_g},
\end{equation}
Thus, in order to estimate $\mu_{k,p}^g$ it is sufficient to estimate the Rayleigh quotient of each of the test functions.

Now we describe explicitly the test functions $u_i$ which we will use in \eqref{minmax2}. Let  $f:[0,\infty)\rightarrow[0,1]$ be defined as follows:
\begin{equation}\label{f}
f(t)=
\begin{cases}
2-2t, & t\in\left[\frac{1}{2},1\right],\\
1, & t\in\left[0,\frac{1}{2}\right],\\
0, & t\in[1,+\infty[.
\end{cases}
\end{equation}

By construction $f\in C^{0,1}[0,+\infty)$. We consider test functions of the form $f(\eta\delta^{g_0}_p(\cdot))$ for some $\eta\in\mathbb R$ and $p\in M$, where $\delta^{g_0}_p(\cdot)$ denotes the Riemannian distance from a point $p\in M$ associated with the metric $g_0$. We note that
\begin{equation}\label{nabla_f}
\nabla_{g_0} f(\eta\delta^{g_0}_p(x))=\eta f'(\eta\delta^{g_0}_p(x))\nabla_{g_0}\delta^{g_0}_p(x).
\end{equation}
Standard computations, and the fact that $|\nabla_{g_0}\delta_p|\leq 1$ a.e. in $M$, show that
\begin{equation}\label{f1}
|\nabla_{g_0}f(\eta\delta^{g_0}_p(x))|\leq 2|\eta|.
\end{equation}
Let now $A_i$ be an annulus of the family $\left\{A_i\right\}_{i=1}^k$. We have two possibilities. Either $A_i$ is a proper annulus with  $0<r_i<R_i$, or is a ball of radius $r_i>0$.
\begin{enumerate}[{\bf Case} \bf a]
\item {\bf (ball).} Assume that $A_i$ is a ball of radius $r_i>0$ and center $p_i$. Associated to $A_i$ we define a function $u_i$ as follows
\begin{equation}\label{ball}
u_i(x)=
\begin{cases}
1, & 0\leq\delta^{g_0}_{p_i}(x)\leq\frac{r_i}{2}\\
f(\frac{\delta^{g_0}_{p_i}(x)}{2r_i}), & r_i\leq\delta^{g_0}_{p_i}(x)\leq 2r_i\\
0, & {\rm otherwise}. 
\end{cases}
\end{equation}
By construction, ${u_i}_{|_{\Omega}}\in W^{1,p}(\Omega)$. Standard computations (see \eqref{nabla_f}-\eqref{f1}) show that 
\begin{equation}\label{grad_ball}
|\nabla_{g_0} u_i|\leq \frac{1}{r_i}.
\end{equation}

This implies that
\begin{equation}\label{cap_ball}
\int_{\Omega\cap 2A_i}|\nabla_{g_0} u_i(x)|^ndv_{g_0}\leq \frac{1}{r_i^n}|\Omega\cap B(p_i,2r_i)|_{g_0}\leq 2^n\omega_n e^{2\kappa(n-1)r_i}\leq 2^n\omega_ne^{n-1},
\end{equation}
where we have used the Bishop-Gromov Theorem and the fact that $2r_i\leq a=\frac{1}{\kappa}$. In particular, Bishop-Gromov Theorem implies that $|B(x,r)|_{g_0}\leq |B(p',r)|_{\kappa}$ where $|B(p',r)|_{\kappa}$ denotes the volume of the ball of radius $r$ in the space form of constant curvature $-\kappa$, and we know that $|B(p',r)|_{\kappa}\leq \omega_n r^ne^{\kappa(n-1)r}$.

\item {\bf (annulus).} Assume that $A_i$ is a proper annulus  of radii $0<r_i<R_i$ and center $p_i$. Associated to $A_i$ we define a function $u_i$ as follows
\begin{equation}\label{prop_annulus}
u_i(x)=
\begin{cases}
1-f(\frac{\delta^{g_0}_{p_i}(x)}{r_i}), & \frac{r_i}{2}\leq\delta^{g_0}_{p_i}(x)\leq r_i\\
1 , & r_i\leq\delta^{g_0}_{p_i}(x)\leq R_i\\
f(\frac{\delta^{g_0}_{p_i}(x)}{2R_i}), & R_i\leq\delta^{g_0}_{p_i}(x)\leq 2R_i\\
0, & {\rm otherwise}.
\end{cases}
\end{equation}
By construction, ${u_i}_{|_{\Omega}}\in W^{1,p}(\Omega)$. Standard computations (see \eqref{nabla_f}-\eqref{f1}) show that
\begin{equation}
|\nabla_{g_0} u_i(x)|\leq 
\begin{cases}
\frac{1}{R_i}, & R_i\leq\delta^{g_0}_{p_i}(x)\leq 2R_i,\\
\frac{2}{r_i}, & \frac{r_i}{2}\leq\delta^{g_0}_{p_i}(x)\leq r_i,\\
0, & {\rm otherwise}.
\end{cases}
\end{equation}
As in \eqref{grad_ball} we have
\begin{equation}\label{grad_ann}
|\nabla_{g_0} u_i|\leq \frac{2}{r_i}.
\end{equation}
Moreover, as for \eqref{cap_ball}, it is possible to prove that
\begin{equation}\label{cap_ann}
\int_{\Omega\cap 2A_i}|\nabla_{g_0} u_i(x)|^ndv_{g_0}\leq \frac{2}{r_i^n}|\Omega\cap B(p_i,2r_i)|_{g_0}\leq 2^{n+1}\omega_ne^{n-1},
\end{equation}
\end{enumerate}

Both \eqref{cap_ball} and \eqref{cap_ann} imply that for any $u_i$ 
\begin{equation}\label{cap_tot}
\int_{\Omega\cap 2A_i}|\nabla_{g_0} u_i(x)|^ndv_{g_0}\leq 2^{n+1}\omega_ne^{n-1},
\end{equation}

Since $1<p\leq n$, using H\"older's inequality and inequality \eqref{cap_tot} we have that
\begin{multline}\label{boundary_est}
\int_{\Omega\cap 2A_i}|\nabla_g u_i|^pdv_g\leq \left(\int_{\Omega\cap 2A_i}|\nabla_g u_i|^ndv_g\right)^{\frac{p}{n}}|\Omega\cap 2A_i|_g^{1-\frac{p}{n}}\\
=\left(\int_{\Omega\cap 2A_i}|\nabla_{g_0} u_i|^ndv_{g_0}\right)^{\frac{p}{n}}|\Omega\cap 2A_i|_g^{1-\frac{p}{n}}
\leq 2^{\frac{(n-1)p}{n}}\omega_n^{\frac{p}{n}}e^{\frac{(n-1)p}{n}}|\Omega\cap 2A_i|_g^{1-\frac{p}{n}}\\
\leq \beta_{p,n}\left(\frac{|\Omega|_g}{k}\right)^{1-\frac{p}{n}},
\end{multline}
where
\begin{equation}\label{alpha_n}
\beta_{p,n}=2^{\frac{(n-1)p}{n}}\omega_n^{\frac{p}{n}}e^{\frac{(n-1)p}{n}}.
\end{equation}
The equality in \eqref{boundary_est} follows from the conformal invariance of the Dirichlet energy, namely $\int_{M}|\nabla_g f|^ndv_g=\int_{M}|\nabla_{g_0} f|^ndv_{g_0}$ for any $g\in[g_0]$.

As for the denominator we have
\begin{equation}\label{denominator}
\int_{\Omega\cap 2A_i}|u_i|^pdv_g\geq \int_{\Omega\cap A_i}|u_i|^pdv_g=|\Omega\cap A_i|_g\geq c\frac{|\Omega|_g}{3k}.
\end{equation}
Then, combining \eqref{boundary_est} and \eqref{denominator} we deduce that for all $i=1,...,k$
\begin{equation}\label{general_bound_laplacian_0}
\frac{\int_{\Omega}|\nabla_g u_i|^pdv_g}{\int_{\Omega}|u_i|^pdv_g}\leq B_{p,n} \left(\frac{k}{|\Omega|_g}\right)^{\frac{p}{n}}
\end{equation}
where $B_{p,n}:=3c^{-1}\beta_{p,n}$. We have proved that for any $k\geq k_{\Omega}$
\begin{equation}\label{large_j}
\mu_{k,p}^g\leq B_{p,n}\left(\frac{k}{|\Omega|_g}\right)^{\frac{p}{n}}.
\end{equation}
\noindent{\bf Step 2 (small $k$).} Let now $k<k_{\Omega}$ be fixed. By using Theorem \ref{corollary_small_annuli} as in Step 1, we find that there exists a sequence of $3k$ sets $\left\{A_i\right\}_{i=1}^{3k}$ such that $|\Omega\cap A_i|_g\geq c\frac{|\Omega|_g}{3k}$. If the sets $A_i$ are annuli, we can proceed as in Step 1 and deduce the validity of \eqref{large_j}. Assume now that $k$ is such that the sets $A_i$ of the decomposition are of the form
$$
A_i=B(x_1^i,r_0)\cup\cdots\cup B(x_{l_i}^i,r_0),
$$
where $r_0=\frac{4a}{1600}$, $D_i=A_i^{4r_0}$ are pairwise disjoint, and $\delta^{g_0}_{x_l^i}(x_j^i)\geq 4r_0$ if $l\ne j$. By definition $D_i=B(x_1^i,5r_0)\cup\cdots\cup B(x_{l_i}^i,5r_0)$. Since we have $3k$ disjoint sets $D_i$, we can pick $k$ of them such that $|\Omega\cap D_i|_g\leq\frac{|\Omega|_g}{k}$ and $|\Omega\cap D_i|_{g_0}\leq\frac{|\Omega|_{g_0}}{k}$. We take from now on this family of $k$ capacitors. Note that $D_i$ is a disjoint union of $l_i$ balls $B(x_1^i,5r_0), \cdots, B(x_{l_i}^i,5r_0)$ of radius $5r_0$. Associated to each $B(x_j^i,5r_0)$, $j=1,...,l$ we construct test functions $u_j^i$ as in \eqref{ball}. Then we define the function $u_i$ associated with the capacitor $(A_i,D_i)$ by setting $u_i=u_j^i$ on $B(x_j^i,5r_0)$. We have $k$ disjointly supported test functions in $W^{1,p}(\Omega)$. We estimate the Rayleigh quotient of each of the $u_i$ as in Step 1. As in \eqref{grad_ball} we estimate $|\nabla_{g_0} u_i|$. In particular, we find a  universal constant $c_0>0$  such that
\begin{equation}\label{est_small_j}
|\nabla_{g_0} u_i|\leq\frac{c_0}{r_0}.
\end{equation}
Moreover, as for \eqref{cap_tot}, it is easy to prove that
\begin{equation}\label{cap_step2}
\int_{\Omega\cap D_i}|\nabla_{g_0}u_i|^ndv_{g_0}\leq\gamma_n,
\end{equation}
where $\gamma_n>0$ depends only on $n$. Then, since $1<p\leq n$, thanks to H\"older's inequality, inequality \eqref{cap_step2}, and analogous computations of those in Step 1, we deduce that
\begin{multline}\label{num_step2}
\int_{\Omega\cap D_i}|\nabla_gu_i|^pdv_g\leq \left(\int_{\Omega\cap D_i}|\nabla_g u_i|^ndv_g\right)^{\frac{p}{n}}|\Omega\cap D_i|_g^{1-\frac{p}{n}}\\
=\left(\int_{\Omega\cap D_i}|\nabla_{g_0} u_i|^ndv_{g_0}\right)^{\frac{p}{n}}|\Omega\cap D_i|_g^{1-\frac{p}{n}}\leq\frac{c_0^p}{r_0^p}|\Omega\cap D_i|_{g_0}^{\frac{p}{n}}|\Omega\cap D_i|_g^{1-\frac{p}{n}}\\
\leq\frac{1600^pc_0^p\kappa^p}{4^p}|\Omega|_{g_0}^{\frac{p}{n}}|\Omega|_g^{1-\frac{p}{n}}k^{-1}.
\end{multline}
As for the denominator
\begin{equation}\label{den_step2}
\int_{\Omega\cap D_i}|u_i|^pdv_g\geq\int_{\Omega\cap A_i}|u_i|^pdv_g=|\Omega\cap A_i|_g\geq c\frac{|\Omega|_g}{3k}.
\end{equation}
From \eqref{num_step2} and \eqref{den_step2} we deduce that
$$
\frac{\int_{\Omega}|\nabla_g u_i|^pdv_g}{\int_{\Omega}|u_i|^pdv_g}\leq A_{p,n}\kappa^p\left(\frac{|\Omega|_{g_0}}{|\Omega|_g}\right)^{\frac{p}{n}},
$$
for all $i=1,...,k$, where $A_{p,n}>0$ depends only on $p$ and $n$. This immediately implies
\begin{equation}\label{small_j}
\mu_{k,p}^g\leq A_{p,n}\kappa^p\left(\frac{|\Omega|_{g_0}}{|\Omega|_g}\right)^{\frac{p}{n}}.
\end{equation}

The proof of \eqref{main_ineq} follows by combining \eqref{large_j} and \eqref{small_j}, possibly re-defining the constants $A_{p,n},B_{p,n}$.

\end{proof}

\begin{rem}
We point out that in the proof of Theorem \ref{main} the inequality \eqref{small_j} appears. Apparently this may look like a nonsense, in fact the right-hand side of the inequality does not depend on $k$. However, note that this situation may occur only for a finite number of eigenvalues $\mu_{p,k}$, since, starting from a certain $k_{\Omega}$ (of which it is possible in principle to give a lower bound), the capacitors of the decomposition given by \eqref{corollary_small_annuli} are of the form $iii)-a)$, hence the estimate \eqref{large_j} holds starting from a certain $k_{\Omega}$.

\end{rem}

\begin{rem}\label{remark_manifold}
We note that if $M$ is a compact manifold, we can choose $\Omega=M$. In this case, the proof of Theorem \ref{main} remains exactly the same. Thus inequality \eqref{main_ineq} holds also for compact manifolds. The left hand side of \eqref{main_ineq} does not depend on $g_0$. Hence, we can take the infimum with respect to $g_0\in[g]$ such that ${\rm Ric}_{g_0}\geq-(n-1)\kappa^2$, $\kappa\geq 0$. From this fact and from Definition \ref{min-conf-vol} we deduce the validity of \eqref{main_ineq_cor}
\end{rem}



We prove now Theorem \ref{main_2}. The proof is similar to that of Theorem \ref{main}, however, though if the differences  may seem slight and technical, they are substantial. In fact, as shown in \cite{matei}, Theorem \ref{main} does not hold for $p>n$.

In the statement and the proof of Theorem \ref{main_2} the metric $g$ is fixed (so everything will be expressed with respect to $g$). In particular, Theorem \ref{corollary_small_annuli} will be exploited with $X=\Omega$ and the Riemannian distance, and the measure induced by the fixed metric $g$. The test functions $u_i$ for the variational principle \eqref{minmax2} will be written in terms of the Riemannian distance associated to $g$, and not in terms of the Riemannian distance associated with some different $g_0\in[g]$ (as in the proof of Theorem \ref{main}).

\begin{proof}[Proof of Theorem \ref{main_2}]
We apply Theorem \ref{corollary_small_annuli} with $a=\frac{1}{\kappa}$ (if $\kappa=0$ we take $a=+\infty$). We take $X=\Omega$ endowed with the  Riemannian distance induced by $g$, and with the measure $\varsigma_g$ defined as the restriction to $\Omega$ of the Lebesgue measure of $M$ associated with the metric $g$, namely $\varsigma_g(E)=|E\cap\Omega|_g$ for all measurable sets $E$.

\noindent{\bf Step 1 (large $k$).} From Lemma \ref{large_n} we deduce that there exists $k_{\Omega}'\in\mathbb N$ such that for all $k\geq k_{\Omega}'$ there exists a sequence $\left\{A_i\right\}_{i=1}^{2k}$ of $2k$ annuli such that $2A_i$ are pairwise disjoint and 
\begin{equation}\label{ineq_gen_1_2}
|\Omega\cap A_i|_g\geq c'\frac{|\Omega|}{2k}.
\end{equation}
The numbers $k_{\Omega}'\in\mathbb N$ and $c'>0$ are in principle different from $k_{\Omega}$ and $c$ of the proof of Theorem \ref{main}. The constant $c'$ depends only on $\Gamma$ of Theorem \ref{corollary_small_annuli}, hence it depends only on the dimension $n$ and can be determined explicitly (see \cite{asma_conf}). Since we have $2k$ annuli, we can pick at least $k$ of them such that
\begin{equation}\label{ineq_gen_2_2}
|\Omega\cap 2A_i|_g\leq \frac{|\Omega|}{k}.
\end{equation}
We take this family of $k$ annuli, and denote it by $\left\{A_i\right\}_{i=1}^k$.


As in the proof of Theorem \ref{main}, subordinated to this decomposition  we construct a family of $k$ disjointly supported functions $u_1,...,u_k\in W^{1,p}(\Omega)$, with ${\rm supp}(u_i)=2A_i$. We also set $F_k:=\left\{\sum_{i=1}^k\alpha_iu_i:\alpha_i\in\mathbb R,\sum_{i=1}^k|\alpha_i|^p=1\right\}$. Inequality \eqref{minmax2} holds, thus, in order to estimate $\mu_{k,p}^g$ it is sufficient to estimate the Rayleigh quotient of each of the test functions. The test functions $u_i$ which we shall use are identical to those in the proof of Theorem \ref{main}, except for the fact that they are given in terms of the distance function associated with the fixed metric $g$.

In particular, if $A_i$ is a ball, the function $u_i$ associated with $A_i$ is given by \eqref{ball}, where we replace $\delta_p^{g_0}(x)$ by $\delta_p^g(x)$, while if $A_i$ is a proper annulus, the associated function $u_i$ is given by \eqref{prop_annulus}, where again we replace $\delta_p^{g_0}(x)$ by $\delta_p^g(x)$. It is straightforward, as in \eqref{grad_ball} and \eqref{grad_ann}, to prove that
\begin{equation}\label{grad_tot}
|\nabla_g u_i(x)|\leq \frac{2}{r_i}.
\end{equation}

Corollary \ref{corollarygny0} gives us information on the size of the radius $r_i$, in fact 
\begin{equation}
r_i\geq\frac{1}{2}\tilde r:=\frac{1}{2}\inf\mathcal B
\end{equation}
where
\begin{equation}
\mathcal B:=\left\{r\in\mathbb R:V(r)\geq \frac{c'|\Omega|_g}{k}\right\}.
\end{equation}
The constant $c'$ depends only on the dimension and is the same constant as in \eqref{ineq_gen_1_2}. We observe that each $r\in\mathcal B$ is such that
$$
\frac{c'|\Omega|_g}{k}\leq V(r)=\sup_{x\in\Omega}|B(x,r)\cap\Omega|_g\leq |B(x,r)|_g\leq |B(p',r)|_{\kappa}
$$
by volume comparison, where $|B(p',r)|_{\kappa}$ denotes the volume of the ball of radius $r$ in the space form of constant curvature $-\kappa$. If $\kappa=0$ then each $r\in\mathcal B$ is such that
$$
\frac{c'|\Omega|_g}{k}\leq \omega_n r^n. 
$$
Hence any $r\in\mathcal B$ is such that
$$
r\geq\left(\frac{c'|\Omega|_g}{\omega_n k}\right)^{\frac{1}{n}},
$$
therefore
\begin{equation}\label{lower_ri}
r_i\geq\frac{1}{2}\left(\frac{c'|\Omega|_g}{\omega_n k}\right)^{\frac{1}{n}}.
\end{equation}
If $\kappa>0$, then $\tilde r\leq 2r_i\leq\frac{1}{\kappa}$ by construction, and since $\tilde r=\inf\mathcal B$, from volume comparison and standard calculus
$$
\frac{c|\Omega|_g}{k}\leq e^{n-1}\omega_n \tilde r^n.
$$
Therefore
\begin{equation}\label{lower_ri_2}
r_i\geq \frac{\tilde r}{2}\geq\frac{1}{2}\left(\frac{c'|\Omega|_g}{e^{n-1}\omega_n k}\right)^{\frac{1}{n}}.
\end{equation}
We note that \eqref{lower_ri} implies \eqref{lower_ri_2} which holds true for any $\kappa\geq 0$. We conclude, by \eqref{ineq_gen_2_2} and \eqref{lower_ri_2}, that
\begin{equation}\label{boundary_est_2}
\int_{\Omega\cap 2A_i}|\nabla_g u_i|^pdv_g\leq \frac{2^p}{r_i^p}|\Omega\cap 2A_i|_g\leq \beta_{p,n}\frac{|\Omega|_g}{k}\left(\frac{k}{|\Omega|_g}\right)^{\frac{p}{n}}
\end{equation}
where
\begin{equation}\label{alpha_n_2}
\beta_{p,n}'=2^{2p}\left(\frac{\omega_n e^{n-1}}{c'}\right)^{\frac{p}{n}}.
\end{equation}

As for the denominator we have
\begin{equation}\label{denominator_2}
\int_{\Omega\cap 2A_i}|u_i|^pdv_g\geq \int_{\Omega\cap A_i}|u_i|^pdv_g=|\Omega\cap A_i|_g\geq c'\frac{|\Omega|_g}{2k}.
\end{equation}
Then, combining \eqref{boundary_est_2} and \eqref{denominator_2} we deduce that for all $i=1,...,k$
\begin{equation}\label{general_bound_laplacian_0_2}
\frac{\int_{\Omega}|\nabla u_i|^pdv_g}{\int_{\Omega}|u_i|^pdv_g}\leq B_{p,n}' \left(\frac{k}{|\Omega|_g}\right)^{\frac{p}{n}}
\end{equation}
where $B_{p,n}':=2(c')^{-1}\beta_{p,n}'$. We have proved that for any $k\geq k_{\Omega}'$
\begin{equation}\label{large_j_2}
\mu_{k,p}^g\leq B_{p,n}'\left(\frac{k}{|\Omega|_g}\right)^{\frac{p}{n}}.
\end{equation}
\noindent{\bf Step 2 (small $k$).} Let now $k<k_{\Omega}'$ be fixed. By using Theorem \ref{corollary_small_annuli} as in Step 1, we find that there exists a sequence of $2k$ sets $\left\{A_i\right\}_{i=1}^{2k}$ such that $|\Omega\cap A_i|_g\geq c'\frac{|\Omega|_g}{2k}$. If the sets $A_i$ are annuli, we can proceed as in Step 1 and deduce the validity of \eqref{large_j_2}. Assume now that $k$ is such that the sets $A_i$ of the decomposition are of the form
$$
A_i=B(x_1^i,r_0)\cup\cdots\cup B(x_{l_i}^i,r_0),
$$
where $r_0=\frac{4a}{1600}$, $D_i=A_i^{4r_0}$ are pairwise disjoint, and $\delta^g_{x_l^i}(x_j^i)\geq 4r_0$ if $l\ne j$. By definition $D_i=B(x_1^i,5r_0)\cup\cdots\cup B(x_{l_i}^i,5r_0)$. Since we have $2k$ disjoint sets $D_i$, we can pick $k$ of them such that $|\Omega\cap D_i|_g\leq\frac{|\Omega|_g}{k}$. We take from now on this family of $k$ capacitors. Note that $D_i$ is a disjoint union of $l_i$ balls $B(x_1^i,5r_0), \cdots, B(x_{l_i}^i,5r_0)$ of radius $5r_0$. Associated to each $B(x_j^i,5r_0)$, $j=1,...,l$ we construct test functions $u_j^i$ as in \eqref{ball} with $\delta^{g_0}$ replaced by $\delta^g$. Then we define the function $u_i$ associated with the capacitor $(A_i,D_i)$ by setting $u_i=u_j^i$ on $B(x_j^i,5r_0)$. We have $k$ disjointly supported test functions in $W^{1,p}(\Omega)$. We estimate the Rayleigh quotient of each of the $u_i$ as in Step 1. As in \eqref{grad_ball} we estimate $|\nabla_g u_i|$. In particular, we find a  universal constant $c_0>0$  such that
\begin{equation}\label{est_small_j_2}
|\nabla_g u_i|\leq\frac{c_0'}{r_0}.
\end{equation}
Analogous computations of those in Step 1 (see \eqref{boundary_est_2} and \eqref{denominator_2}) allow us to conclude that
$$
\frac{\int_{\Omega}|\nabla u_i|^pdv_g}{\int_{\Omega}|u_i|^pdv_g}\leq \frac{A_{p,n}'}{a^p}=A_{p,n}\kappa^p,
$$
for all $i=1,...,k$, where $A_{p,n}'>0$ depends only on $p$ and $n$. This immediately implies
\begin{equation}\label{small_j_2}
\mu_{k,p}\leq A_{p,n}'\kappa^p.
\end{equation}

The proof of \eqref{main_ineq_2} follows by combining \eqref{large_j_2} and \eqref{small_j_2}, possibly re-defining the constants $A_{p,n}',B_{p,n}'$.

\end{proof}

\begin{rem}\label{remark_manifold_2}
We note that if $M$ is a compact manifold, we can choose $\Omega=M$. In this case, the proof of Theorem \ref{main_2} remains exactly the same. Thus inequality \eqref{main_ineq_2} holds also for compact manifolds. This is exactly the statement of Theorem \ref{cor_2}.
\end{rem}

\appendix\section{Isoperimetric bounds for compact hypersurfaces}\label{appendix}

Following \cite{colboisgirouard}, we observe that Lemma \ref{genCM-cor} allows to prove a number of further upper bounds for the eigenvalues of the $p$-Laplacian on a compact hypersurface in terms of the isoperimetric ratio. For the sake of completeness, we collect the results here. 

First, we need to introduce some notation. Let $(M,g)$ be a complete $n$-dimensional Riemannian manifold and let $\Omega$ be a bounded domain of $M$ with smooth boundary $\Sigma:=\partial\Omega$. We denote by $I(\Omega)$ the isoperimetric ratio of $\Omega$, namely
\begin{equation}\label{isop_ratio}
I(\Omega):=\frac{|\Sigma|}{|\Omega|^{\frac{n-1}{n}}}.
\end{equation}
We also denote by $I_0(\Omega)$ the quantity
\begin{equation}\label{isop_ratio_0}
I_0(\Omega):=\inf\{I(U):U\subset\Omega{\rm\ open\ }\}.
\end{equation}
For any $x\in M$ we denote by $r(x)$ the supremum of those $r\geq 0$ such that, for all $s\leq r$ one has both $|B(x,s)|\leq 2\omega_ns^n$ and $|\partial B(x,s)|\leq 2n\omega_n s^{n-1}$ (note that if ${\rm Ric}_g\geq 0$, then $r(x)=+\infty$ for all $x\in M$). We denote by $r_-(\Omega)$ the quantity
\begin{equation}\label{r-}
r_-(\Omega):=\inf_{x\in\Sigma}r(x).
\end{equation}
Finally, for any $r>0$ we denote by $N_M(r)$ an integer number such that for any $x\in M$ and $s<r$, the geodesic ball $B(x,5s)$ can be covered by $N_M(r)$ balls of radius $s$.

It is not in general easy to estimate $I_0(\Omega)$, $N_M(r)$ or $r_-(\Omega)$. However, having information such as a lower bound on ${\rm Ric}_g$ allows to control these quantities in a nice way. For example, if $M=\mathbb R^n$ with the Euclidean metric, and $\Omega$ is any smooth domain, then $I_0(\Omega)=n \omega_n^{\frac{1}{n}}$ (this follows from the isoperimetric inequality in $\mathbb R^n$), $N_M(r)\leq 40^n$ for any $r>0$ (this is a consequence of a standard packing lemma, see e.g., \cite[Lemma 3.1]{colboisgirouard} and references therein), and $r_-(\Omega)=+\infty$.

We are ready to state the following result, which is the key ingredient for the proof of the upper bounds for the eigenvalues of the $p$-Laplacian on hypersurfaces.

\begin{prop}\label{prop_hyp}
Let $(M,g)$ be a complete $n$-dimensional Riemannian manifold and let $\Omega$ be a bounded domain of $M$ with smooth boundary $\Sigma$. Let $\mu_{k,p}(\Sigma)$ denote the variational eigenvalues of the $p$-Laplacian on $\Sigma$ with the induced Riemannian metric. Let $0<r_0<r_-(\Omega)$ and let $k_0$ be the first integer to satisfy
$$
k_0>\frac{1}{25n\omega_n}\frac{I_0(\Omega)}{r_0^{n-1}}|\Omega|^{\frac{n-1}{n}}.
$$
Then for all $k\geq k_0$ we have
\begin{equation}\label{ineq_prop_hyp}
\mu_{k,p}(\Sigma)\leq 625(25n\omega_n)^{\frac{p}{n-1}}N_M(r_0)^2\left(\frac{I(\Omega)}{I_0(\Omega)}\right)^{1+\frac{p}{n-1}}\left(\frac{k}{|\Sigma|}\right)^{\frac{p}{n-1}}.
\end{equation}
\end{prop}

The proof of Proposition \ref{prop_hyp} can be carried out in the same way as that of \cite[Proposition 2.1]{colboisgirouard} (see also \cite[Theorem 1.5]{colbois_gittins}). In particular it makes use of Lemma \ref{genCM-cor} here above. Proposition \ref{prop_hyp} allows to prove the following theorem.

\begin{thm}\label{thm_hyp}
Let $(M,g)$ be a complete $n$-dimensional Riemannian manifold and let $\Omega$ be a bounded domain of $M$ with smooth boundary $\Sigma$. Let $\mu_{k,p}(\Sigma)$ denote the variational eigenvalues of the $p$-Laplacian on $\Sigma$ with the induced Riemannian metric. Then for any $r_0<\frac{r_-(\Omega)}{5}$ we have
\begin{equation}\label{ineq_thm_hyp}
\mu_{k,p}(\Sigma)\leq 625N_M(r_0)^2\frac{I(\Omega)}{I_0(\Omega)}\left[\frac{1}{r_0^p}+\left(25n\omega_n\frac{I(\Omega)}{I_0(\Omega)}\frac{k}{|\Sigma|}\right)^{\frac{p}{n-1}}\right],
\end{equation}
for all $k\in\mathbb N$
\end{thm}
\begin{proof}
If $k\geq k_0$ this follows immediately from Proposition \eqref{prop_hyp}. If $k<k_0$, then $\mu_{k,p}(\Sigma)\leq\mu_{k_0,p}(\Sigma)$. Then, from Proposition \eqref{prop_hyp} it follows that
\begin{equation}\label{proof0}
\mu_{k,p}(\Sigma)\leq 625(25n\omega_n)^{\frac{p}{n-1}}N_M(r_0)^2\left(\frac{I(\Omega)}{I_0(\Omega)}\right)^{1+\frac{p}{n-1}}\frac{\max\{k_0^{\frac{p}{n-1}},k^{\frac{p}{n-1}}\}}{|\Sigma|^{\frac{p}{n-1}}}.
\end{equation}
Since $k<k_0$, then $k\leq k_0-1$ and
\begin{equation}\label{proof1}
\max\{k_0^{\frac{p}{n-1}},k^{\frac{p}{n-1}}\}\leq (k_0-1)^{\frac{p}{n-1}}+k^{\frac{p}{n-1}}.
\end{equation}
Moreover
\begin{equation}\label{proof2}
k_0-1\leq\frac{1}{25n\omega_n}\frac{I_0(\Omega)}{r_0^{n-1}}|\Omega|^{\frac{n-1}{n}}=\frac{1}{25n\omega_n}\frac{I_0(\Omega)}{r_0^{n-1}}\frac{|\Sigma|}{I(\Omega)}.
\end{equation}
Inequality \eqref{ineq_thm_hyp} for $k<k_0$ follows by combining \eqref{proof0}, \eqref{proof1} and \eqref{proof2}. This concludes the proof.
\end{proof}

As a corollary of Theorem \ref{thm_hyp} we have the following result on hypersurfaces of the Euclidean space.

\begin{cor}\label{cor_hyp_eucl}
For any bounded domain $\Omega\subset\mathbb R^n$ with smooth boundary $\Sigma$ we have
\begin{equation}\label{final_eucl}
\mu_{k,p}(\Sigma)\leq C_{p,n} I(\Omega)^{1+\frac{p}{n-1}}\left(\frac{k}{|\Sigma|}\right)^{\frac{p}{n-1}},
\end{equation}
for all $k\in\mathbb N$, where $C_{p,n}$ depends only on $p$ and $n$.
\end{cor}
\begin{proof}
The result follows immediately from Theorem \ref{thm_hyp}. In fact, in the case of $\mathbb R^n$ we easily check that $I_0(\Omega)=n\omega_n^{\frac{1}{n}}$ and $N_M(r)\leq 40^n$ for all $r>0$. Moreover, we can let $r_0\rightarrow+\infty$, being $r_-(\Omega)=+\infty$.
\end{proof}

Theorem \ref{thm_hyp} also implies the following corollary for hypersurfaces bounding some domain $\Omega$ in a Riemannian manifold $(M,g)$, when we assume a lower bound on the Ricci curvature of $(M,g)$.

\begin{cor}\label{cor_hyp}
Let $(M,g)$ be a complete $n$-dimensional Riemannian manifold with ${\rm Ric}_g\geq -(n-1)\kappa^2$, $\kappa\geq 0$. For any bounded domain $\Omega$ with smooth boundary $\Sigma$ we have
\begin{equation}\label{final}
\mu_{k,p}(\Sigma)\leq A_{p,n} \frac{I(\Omega)}{I_0(\Omega)}\kappa^p+B_{p,n}\left(\frac{I(\Omega)}{I_0(\Omega)}\right)^{1+\frac{p}{n-1}}\left(\frac{k}{|\Sigma|}\right)^{\frac{p}{n-1}},
\end{equation}
for all $k\in\mathbb N$, where $A_{p,n},B_{p,n}$ depend only on $p$ and $n$.
\end{cor}
\begin{proof}
In the case that $\kappa=0$, the proof is essentially the same as in the Euclidean case (see Corollary \ref{cor_hyp_eucl}). In fact from Bishop-Gromov Theorem we deduce that $N_M(r)\leq 40^n$ for all $r>0$, and we can let $r_0\rightarrow +\infty$ being $r_-(\Omega)=+\infty$. As for $\kappa\ne 0$, it is sufficient to prove \eqref{final} for $\kappa=1$, noting that the result for any $\kappa\ne 0$ follows by replacing the metric $g$ on $M$ by $\kappa^2g$, and considering on $\Sigma$ the metric induced by $\kappa^2g$. Then ${\rm Ric}_{\kappa^2g}\geq-(n-1)$, $\mu_{k,p}(\Sigma,\kappa^2g)=\frac{1}{\kappa^p}\mu_{k,p}(\Sigma)$ and $|(\Sigma,\kappa^2g)|=\kappa^{n-1}|\Sigma|$, while the isoperimetric ratio is invariant under scaling.

When $\kappa=1$, let $r(n)>0$ denote the largest $r>0$ such that $(\sinh(r))^{n-1}\leq 2r^{n-1}$ for all $r<r(n)$. The constant $r(n)$ depends only on $n$. Since, by volume comparison, the volume of a ball of radius $r$ in $(M,g)$ is bounded by $n\omega_n\int_0^r(\sinh(s))^{n-1}ds$ (the volume in the standard hyperbolic space) and the $n-1$-dimensional volume of the sphere of radius $r$ in $(M,g)$ is bounded by $n\omega_n(\sinh(r))^{n-1}$ (the corresponding volume in the standard hyperbolic space), we deduce that $r_-(\Omega)\geq r(n)>0$. Moreover, still Bishop-Gromov volume comparison Theorem  allows  to conclude that there exists a constant $V(n)$ depending only on $n$ such that $N_M(r)\leq V(n)$ for all $r\leq r(n)$. Applying Theorem \ref{thm_hyp} we conclude that 
$$
\mu_{k,p}(\Sigma)\leq A_{p,n} \frac{I(\Omega)}{I_0(\Omega)}+B_{p,n}\left(\frac{I(\Omega)}{I_0(\Omega)}\right)^{1+\frac{p}{n-1}}\left(\frac{k}{|\Sigma|}\right)^{\frac{p}{n-1}},
$$
for all $k\in\mathbb N$, for some constants $A_{p,n}$, $B_{p,n}$ depending only on $p$ and $n$.
This concludes the proof.
\end{proof}

\begin{rem}
We remark that in this case the isoperimetric ratio $I(\Omega)$ in \eqref{final_eucl} and \eqref{final} cannot be decoupled, in principle, from $k$. This has been observed for $p=2$ in \cite{luc}.
\end{rem}

\section*{Acknowledgements}
The second author acknowledges hospitality of the Institut de Math\'emathiques of the University of Neuch\^atel offered during the development of part the work. The second author is member of the Gruppo Nazionale per l'Analisi Matematica, la Probabilit\`a e le loro Applicazioni (GNAMPA) of the Istituto Nazionale di Alta Matematica (INdAM).


\bibliography{bibliography}{}
\bibliographystyle{abbrv}


\end{document}